\documentclass[letterpaper,11pt]{amsart}
\usepackage{graphicx, amscd, amssymb, young}

\usepackage{epsfig}
\usepackage{color}

\newtheorem{lem}{Lemma}[section]
\newtheorem{thm}[lem]{Theorem}
\newtheorem{pro}[lem]{Proposition}
\newtheorem{cor}[lem]{Corollary}
\newtheorem{exa}[lem]{Example}

\newcommand{\M}{{\mathcal{M}}}

\newcommand{\U}{{\textsf{U}}}
\newcommand{\D}{{\textsf{D}}}
\renewcommand{\L}{{\textsf{L}}}

\newcommand{\F}{{\mathcal{F}}}
\newcommand{\R}{{\mathcal{R}}}
\newcommand{\T}{{\mathcal{T}}}
\newcommand{\Q}{{\mathcal{Q}}}

\newcommand{\W}{{\mathcal{W}}}
\newcommand{\X}{{\mathcal{X}}}

\newcommand{\RS}{{\mathcal{RS}}}
\newcommand{\DRS}{{\mathcal{DRS}}}

\begin{document}
\title[Simsun and Double Simsun Permutations]{On Simsun and Double Simsun Permutations Avoiding a Pattern of Length Three}
\author{Wan-Chen Chuang}
\address{Department of Applied Mathematics, National University of Kaohsiung, Kaohsiung 811, Taiwan, ROC}
\email{m0974103@mail.nuk.edu.tw}

\author{Sen-Peng Eu}
\address{Department of Applied Mathematics, National University of Kaohsiung, Kaohsiung 811, Taiwan, ROC}
\email{speu@nuk.edu.tw}

\author{Tung-Shan Fu}
\address{Mathematics Faculty, National Pingtung Institute of Commerce, Pingtung 900, Taiwan, R.O.C}
\email{tsfu@npic.edu.tw}

\author{Yeh-Jong Pan}
\address{Department of Computer Science and Information Engineering, Tajen University, Pingtung 907, Taiwan, R.O.C}
\email{yjpan@mail.tajen.edu.tw}

\thanks{S.-P. Eu is partially supported by National Science Council, Taiwan under grants NSC
98-2115-M-390-002-MY3, T.-S. Fu is partially supported by NSC
97-2115-M-251-001-MY2, Y.-J. Pan is partially supported by NSC
98-2115-M-127-001}

\begin{abstract}
A permutation $\sigma\in\mathfrak{S}_n$ is simsun if for all $k$, the subword of $\sigma$ restricted to $\{1,\dots,k\}$ does not have three consecutive decreasing elements. The permutation $\sigma$ is double simsun if both $\sigma$ and $\sigma^{-1}$ are simsun. In this paper we present a new bijection between simsun permutations and increasing 1-2 trees, and show a number of interesting consequences of this bijection in the enumeration of pattern-avoiding simsun and double simsun permutations. We also enumerate the double simsun permutations that avoid each pattern of length three.
\end{abstract}

\maketitle

\tableofcontents


\section{Introduction}

\subsection{Simsun and double simsun permutations}
For a permutation $\sigma=\sigma_1\cdots\sigma_n\in\mathfrak{S}_n$, a {\em descent} of $\sigma$ is a pair $(\sigma_i,\sigma_{i+1})$ of adjacent elements with $\sigma_i>\sigma_{i+1}$ ($1\le i\le n-1$), and a {\em double descent} of $\sigma$ is a triple $(\sigma_i,\sigma_{i+1},\sigma_{i+2})$ of consecutive elements with $\sigma_i>\sigma_{i+1}>\sigma_{i+2}$ ($1\le i\le n-2$). The permutation $\sigma$ is called {\em simsun} if for all $k$, the subword of $\sigma$ restricted to $\{1,\dots,k\}$ (in the order they appear in $\sigma$) has no double descents. For example, $\sigma=24351$ is not simsun since when restricted to $\{1,2,3,4\}$ the subword $2431$ of $\sigma$ contains a double descent 431.

Simsun permutations were named after Rodica Simion and Sheila Sundaram \cite{Sund}. They are a variant of  Andr\'{e} permutations of Foata and Sch\"{u}tzenberger \cite{FoataSchu}, and are related to the enumeration of the monomials of the $cd$-index of $\mathfrak{S}_n$ (see \cite{Hety,HetyRein}).
Chow and Shiu \cite{ChowShiu} enumerated simsun permutations by descent, using generating functions.
Let $\RS_n$ denote the set of simsun permutations in $\mathfrak{S}_n$. Simion and Sundaram proved that
\begin{equation}
|\RS_n|=E_{n+1},
\end{equation}
where $E_n$ is the $n$th {\em Euler number}, which also counts the number of permutations $\sigma\in\mathfrak{S}_n$ with the property $\sigma_1>\sigma_2<\sigma_3>\sigma_4<\cdots$, known as {\em alternating permutations}.

Inspiring by the notion of double alternating permutations proposed
by Stanley \cite{Stan}, we call a permutation
$\sigma\in\mathfrak{S}_n$ {\em double simsun} if both $\sigma$ and
$\sigma^{-1}$ are simsun. For example, $\sigma=51324$ is simsun but
not double simsun since $\sigma^{-1}=24351\not\in\RS_5$.

\smallskip
\subsection{Pattern-avoiding simsun and double simsun permutations}
Recently, Deutsch and Elizalde \cite{DeutEliz} enumerated simsun permutations that avoid a pattern or a set of patterns of length 3.
For an integer $t\le n$, let $\omega=\omega_1\cdots\omega_t\in\mathfrak{S}_t$. We say that $\sigma$ contains an $\omega$-{\em pattern} if there are indices $i_1<i_2<\cdots <i_t$ such that
$\sigma_{i_j}<\sigma_{i_k}$ if and only if $\omega_j<\omega_k$. Moreover, $\sigma$ is called $\omega$-{\em avoiding} if $\sigma$ contains no $\omega$-patterns.
Let $\RS_n(\omega)$ denote the set of $\omega$-avoiding permutations in $\RS_n$. One of Deutsch and Elizalde's results \cite{DeutEliz} is the complete enumeration of $\RS_n(\omega)$, for any $\omega\in\mathfrak{S}_3$. The counting numbers are listed in the second column in Table \ref{tab:S3-avoiding}. Some results involve classical numbers such as Catalan number $C_n$, Motzkin number $M_n$, secondary structure number $S_n$, and Fibonacci number $F_n$ (e.g., $|\RS_n(132,213)|=F_{n+1}$, see \cite{DeutEliz}).

\begin{table}
\begin{tabular}{|c|c|c|}
\hline
$\omega$ & $|\RS_n(\omega)|$  & $|\DRS_n(\omega)|$ \\
\hline
123 & 6 (for $n\ge 4$) & 2 (for $n\ge 6$)\\
132 & $S_n$ & $S_n$\\
213 & $M_n$ & $S_n$\\
231 & $M_n$ & $2^{n-1}$\\
312 & $2^{n-1}$ & $2^{n-1}$\\
321 & $C_n$ & $C_n$\\
\hline
\end{tabular}
\vspace{0.1in}
\caption{\small The number of simsun and double simsun permutations avoiding a pattern of length 3.}
\label{tab:S3-avoiding}
\end{table}

In this paper we study the enumeration of pattern-avoiding double simsun permutations. For an $\omega\in\mathfrak{S}_t$, the permutation $\sigma$ is called $\omega$-{\em avoiding double simsun} if $\sigma$ is $\omega$-avoiding simsun and $\sigma^{-1}$ is simsun. Let $\DRS_n(\omega)$ be the set of $\omega$-avoiding double simsun permutations in $\mathfrak{S}_n$. Note that in this case $\sigma^{-1}$ is not necessarily $\omega$-avoiding.  For example, $231\in\DRS_3(312)$ since $231\in\RS_3(312)$ and $(231)^{-1}=312\in\RS_n$.
One of the main results is the following enumerative hierarchy for restricted simsun permutations
\begin{equation} \label{eqn:hierarchy}
\DRS_n(132,213)\subseteq\DRS_n(213)\subseteq\RS_n(213)\subseteq\RS_n,
\end{equation}
where $|\DRS_n(132,213)|=F_{n+1}$, $|\DRS_n(213)|=S_n$, and
$|\RS_n(213)|=M_n$. In particular, we characterize the permutations in $\DRS_n(213)$ among the permutations in $\RS_n(213)$ by a pattern-condition (Theorem \ref{thm:characterization}).
Moreover, we give a unified approach to prove these results based on a bijection between simsun permutations and increasing 1-2 trees.

\smallskip
\subsection{Increasing 1-2 trees}
A rooted tree on the vertex set $[0,n]:=\{0,1,\dots,n\}$ is {\em increasing} if every path from the root is increasing. The vertices with no children are called {\em leaves}, and the other vertices are called {\em inner nodes}.
Let $\T_n$ denote the set of increasing trees on $[0,n]$ such that every vertex has at most two children. (The order of the subtrees of a vertex is irrelevant.)  Members of $\T_n$ are called {\em increasing 1-2 trees} on $[0,n]$. For example, the five trees in $\T_3$ are shown in Figure \ref{fig:Euler-4}.

\begin{figure}[ht]
\begin{center}
\includegraphics[width=2.8in]{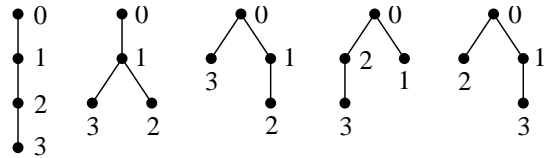}
\end{center}
\caption{\small The increasing 1-2 trees with four vertices.} \label{fig:Euler-4}
\end{figure}

Increasing 1-2 trees appeared in connection to the enumeration of alternating permutations (e.g., see \cite{Dona, KuznParPost}).
There is a known bijection between increasing 1-2 trees and simsun permutations, due to Maria Monks (mentioned in \cite[Solution to Exercise 120]{Stan-2}), which is given in terms of flip equivalence classes of increasing binary trees. In this paper we present a new bijection $\phi:\T_n\rightarrow\RS_n$ (Theorem \ref{thm:phi}), which has a number of interesting consequences in the enumeration of pattern-avoiding simsun/double simsun permutations. We also enumerate the double simsun permutations that avoid each pattern of length 3.

This paper is organized as follows. The bijection $\phi:\T_n\rightarrow\RS_n$ is given in section 2. With the bijection $\phi$ restricted to (unlabeled) ordered 1-2 trees we enumerate the sets in the hierarchy (\ref{eqn:hierarchy}), $\RS_n(231)$, and $\RS_n(231,213)$ in section 3. The enumeration of $\DRS_n(\omega)$, for $\omega\in\{132,312,231,321,123\}$, and $\DRS_n(312,231)$ is given in section 4. In particular, we give simple constructions for the permutations in $\DRS_n(312)$ and $\DRS_n(231)$.


\section{A bijection between simsun permutations and increasing 1-2 trees}
In this section we prove the following theorem.

\begin{thm} \label{thm:phi}
There is a bijection $\phi:\T_n\rightarrow\RS_n$ such that a tree $T\in\T_n$ with $k+1$ leaves is carried to a permutation $\phi(T)\in\RS_n$ with $k$ descents.
\end{thm}

Given a $T\in\T_n$, we write $T$ in a {\em canonical form} such that if a vertex $x$ has two children $u,v$ with $u>v$ then $u$ is the {\em left child}, $v$ is the {\em right child}. The vertices $u,v$ are {\em siblings}. We make the convention that if $x$ has only one child then it is the right child of $x$.
For two vertices $x,y\in T$, we say that $y$
is a {\em descendant} of $x$ if $x$ is contained in the path from
$y$ to the root. Let $\tau(x)$ denote the subtree of
$T$ consisting of $x$ and the descendants of $x$, and let $T-\tau(x)$
denote the subgraph of $T$ when $\tau(x)$ is removed.

\medskip
\subsection{The bijection $\phi$} Given
a set $X\subseteq [n]$ and an increasing 1-2 tree $T$ on $X\cup\{0\}$, we associate $T$ with a word $\phi(T)$ of length $|X|$ with alphabet $X$ and without repeated letters by the following algorithm. By the {\em inorder} traversal of a tree we mean visiting the left subtree (possibly empty), the root, and then the right subtree, recursively.

\smallskip
\noindent{\bf Algorithm A.}
\begin{enumerate}
\item[(A1)] If $T$ consists of the root vertex then $T$ is associated with an empty word.
\item[(A2)] Otherwise the word $\phi(T)$ is defined inductively by the factorization
\[
\phi(T)=\omega\cdot\phi(T'),
\]
where the subword $\omega$ and the subtree $T'$ are determined as follows.
\end{enumerate}
\begin{itemize}
  \item If the root of $T$ has only one child $x$ then let $\omega=x$ (consisting of a single letter $x$) and let $T'=\tau(x)$ (i.e, obtained from $T$ by deleting the root of $T$), and relabel the vertex $x$ by 0.
  \item If the root of $T$ has two children $u,v$ with $u>v$ then traverse the left subtree $\tau(u)$ in inorder and write down the word $\omega$ of the vertices of $\tau(u)$. Let $T'=T-\tau(u)$ (i.e., obtained by removing $\tau(u)$ from $T$).
\end{itemize}

\smallskip
\begin{exa} \label{exa:algorithm-A} {\rm
Let $T$ be the tree shown in Figure \ref{fig:simsun}(a), which is in the canonical form. Since the root of $T$ has two children, the word $\phi(T)$ can be factorized as $\phi(T)=\omega_1\cdot\phi(T_1)$, where $\omega_1=3 8 4 6$ is the inorder of the vertices of the left subtree $\tau(3)$, and $T_1=T-\tau(3)$  shown as Figure \ref{fig:simsun}(b). Since the root of $T_1$ has only one child, $\phi(T_1)$ can be further factorized as $\phi(T_1)=\omega_2\cdot\phi(T_2)$, where $\omega_2=1$ and $T_2=\tau(1)$ shown as Figure \ref{fig:simsun}(c). Inductively, $\phi(T_2)=9 5 7 2$. We then obtain the corresponding word $\phi(T)= 3 8 4 6 1 9 5 7 2$.
}
\end{exa}

\begin{figure}[ht]
\begin{center}
\includegraphics[width=4in]{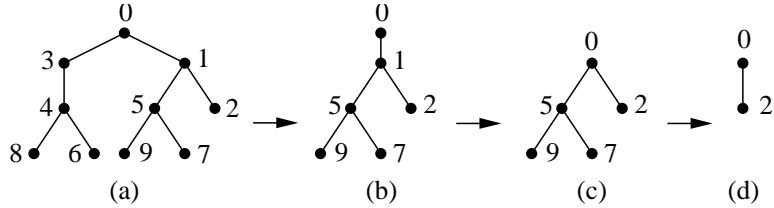}
\end{center}
\caption{\small The iterative stages of the bijection $\phi$ for Example \ref{exa:algorithm-A}.} \label{fig:simsun}
\end{figure}

\begin{pro} For every $T\in\T_n$, the word $\phi(T)$ is a simsun permutation in $\mathfrak{S}_n$.
\end{pro}

\begin{proof} We shall prove that for all $k$, the subword $\omega$ of $\phi(T)$ restricted to $\{1,\dots,k\}$ contains no double descents. Note that
such a word $\omega$ is the word associated with the subtree of $T$ obtained by removing the vertices $n,n-1,\dots,k+1$ from $T$.
Thus it suffices to prove that $\phi(T)$ contains no double descents for any $T\in\T_n$ and for all $n$. We proceed by induction on the number of vertices of $T$. For $n=1$ and 2, the two cases are trivial. For $n\ge 3$, we distinguish the following two cases.

(i) The root of $T$ has two children. Let $u$ be the left child of the root. Note that the right child of the root must be the vertex 1.)
Since $\omega$ is the inorder of the vertices in $\tau(u)$, a pair $(\omega_i,\omega_{i+1})\subseteq\omega\cup\{1\}$ is a descent whenever the vertex $\omega_i$ is a leaf in $\tau(u)$. Moreover, within the inorder, there is at least one inner node between any two leaves. It follows that there are no double descents in $\omega\cup\{1\}$. The remaining part of $\phi(T)$ can be checked as in case (ii).

(ii) The root of $T$ has only one child. Let $0=x_0,x_1,\dots,x_t=x$ be the path that connects the root and the vertex $x$, where $x_j$ is the only child of $x_{j-1}$ ($1\le j\le t$) and $x$ is the first vertex in depth-first-search order of $T$ that has two children. Then $\phi(T)$ can be factorized as $\phi(T)=\mu_1\mu_2$, where $\mu_1=1\cdots t$ is the initial subword and $\mu_2$ is the remaining part. Clearly, $\mu_1$ has no descents. The subword $\mu_2$, which is determined by the subtree $\tau(x)$, can be checked as in case (i).

By induction we prove that $\phi(T)$ contains no double descents. The assertion follows.
\end{proof}

\medskip
\subsection{Finding $\phi^{-1}$}
For a tree $T\in\T_n$ and $i\in [n]$, let $V(i)$ denote the vertex $i$ in $T$.  By the {\em rightmost path} of $T$ we mean the path from the root to the last vertex in depth-first-search order of $T$.

Given a $\sigma\in\RS_n$, we shall recover the tree $\phi^{-1}(\sigma)$ by constructing a sequence of trees $T_1, T_2,\dots,T_n=\phi^{-1}(\sigma)$, where $T_i$ is obtained from $T_{i-1}$ by attaching the vertex $V(i)$ to some vertex $u$ of $T_{i-1}$ so that $V(i)$ is a child of $u$. In fact, $T_i$ corresponds to the subword of $\sigma$ restricted to $\{1,\dots,i\}$.

\smallskip
\noindent{\bf Algorithm B.}

Initially $T_1$ is the tree with $V(1)$ attached to the root $0$.
Suppose we have constructed up to $T_{j-1}$ for some $j\ge 2$. Let $\omega=\omega_1\cdots\omega_{j-1}$ be the subword of $\sigma$ restricted to $\{1,\dots,j-1\}$. To construct $T_j$, we add  $V(j)$ to $T_{j-1}$ according to the following cases.
\begin{enumerate}
  \item[(B1)] The element $j$ appears after $\omega_{j-1}$ in $\sigma$. Then we attach $V(j)$ to the last vertex (in depth-first-search order) of $T_{j-1}$.
  \item[(B2)] The element $j$ appears before $\omega_1$ in $\sigma$. If the root of $T_{j-1}$ has only one child then we attach $V(j)$ to the root, otherwise we attach $V(j)$ to $V(\omega_1)$.
  \item[(B3)] The element $j$ is between $\omega_{i-1}$ and $\omega_i$ in $\sigma$, for some $i\le j-1$. There are two cases.
 \begin{itemize}
   \item[(a)] $\omega_{i-1}>\omega_i$. We attach $V(j)$ to  $V(\omega_{i-1})$.
   \item[(b)] $\omega_{i-1}<\omega_i$. If $V(\omega_i)$ is in the rightmost path of $T_{j-1}$ then we attach $V(j)$ to $V(\omega_{i-1})$, otherwise we attach $V(j)$ to $V(\omega_i)$.
 \end{itemize}
\end{enumerate}

\smallskip
\begin{exa}{\rm Take $\sigma=5 3 4 1 8 6 7 2\in\RS_8$. The sequence $T_1,T_2,\dots,T_8$ of trees constructed for $\phi^{-1}(\sigma)=T_8$ is shown in Figure \ref{fig:1-2-trees}. Note that $V(2)$ is attached to $V(1)$ (shown as $T_2$) since 2 appears after 1 in $\sigma$. For $T_3$, $V(3)$ is attached to the root 0 (shown as $T_3$) since $3$ appears before the subword 12 in $\sigma$ and the root of $T_2$ has only one child. For $T_4$, $V(4)$ is attached to $V(3)$ since 4 appears between 3 and 1. For $T_5$, $V(5)$ is attached to $V(3)$ since 5 appears before the subword 3412 in $\sigma$ and the root of $T_4$ has two children. For $T_6$, $V(6)$ is attached to $V(1)$ since 6 appears between 1 and 2, and $V(2)$ is in the rightmost path. For $T_7$, $V(7)$ ia attached to $V(6)$ since 7 appears between 6 and 2. For $T_8$, $V(8)$ is attached to $V(6)$ since 8 appears between 1 and 6, and $V(6)$ is not in the rightmost path.
}
\end{exa}

\begin{figure}[ht]
\begin{center}
\scalebox{1.25}{\input{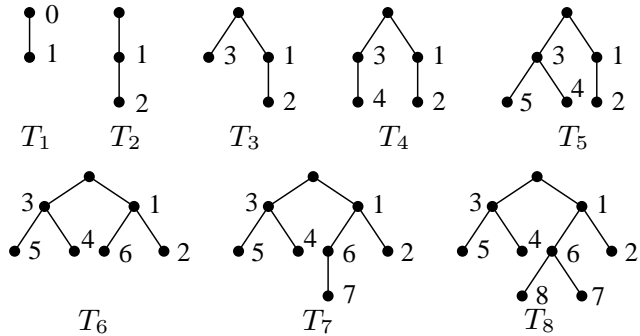}}
\end{center}
\caption{\small The sequence of increasing 1-2 trees for $\sigma=5 3 4 1 8 6 7 2$.} \label{fig:1-2-trees}
\end{figure}

\begin{pro} For every $\sigma\in\RS_n$, the tree $\phi^{-1}(\sigma)$ is an increasing 1-2 tree in $\T_n$.
\end{pro}

\begin{proof} Let $T_1,\dots,T_n$ be the sequence of trees constructed for $\phi^{-1}(\sigma)=T_n$. It is clear that these trees are increasing, and that $T_j$ corresponds to the subword $\mu$ of $\sigma$ restricted to $\{1,\dots,j\}$, for all $j$. We claim that each $T_j$ is an 1-2 trees and $\mu=\phi(T_j)$.

For $T_1$, it is trivial. Suppose the assertion holds up to $T_{j-1}$, for some $j\ge 2$. Let $\phi(T_{j-1})=\omega_1\cdots\omega_{j-1}$, i.e., the subword of $\sigma$ restricted to $\{1,\dots,j-1\}$.

For $T_j$, it suffices to show that the vertex $V(j)$ is attached to a vertex that has at most one child in $T_{j-1}$. The cases (B1) and (B2) of algorithm B are clear. For case (a) of (B3), since $(\omega_{i-1},\omega_i)$ is a descent, $V(\omega_{i-1})$ is a leaf in $T_{j-1}$. For case (b) of (B3), the element $j$ appears between $\omega_{i-1}$ and $\omega_i$ in $\sigma$ and $\omega_{i-1}<\omega_i$. By algorithm A and the fact that $\omega_{i-1}$ and $\omega_i$ are consecutive in $\phi(T_{j-1})$, we observe that
if $V(\omega_i)$ is in the rightmost path then $V(\omega_i)$ is the only child of $V(\omega_{i-1})$ in $T_{j-1}$ (otherwise the left child of $V(\omega_{i-1})$ will appear between $\omega_{i-1}$ and $\omega_i$ in $\phi(T_{j-1})$). Moreover,
if $V(\omega_i)$ is not in the rightmost path then $V(\omega_i)$ has at most one child in $T_{j-1}$ (otherwise due to the vertex-inorder of the subtree $\tau(\omega_i)$, the left child of $V(\omega_i)$ will appear between $\omega_{i-1}$ and $\omega_i$ in $\phi(T_{j-1})$).
Hence $T_j$ is 1-2 tree. It is straightforward to show that $\phi(T_j)=\mu$. The proof is completed.
\end{proof}

\smallskip
Note that $(\omega_i,\omega_{i+1})$ is a descent $\phi(T)$ if and only if the vertex $V(\omega_i)\in T$ is a leaf other than the last vertex. This completes the proof of Theorem \ref{thm:phi}.

\smallskip
\noindent{\bf Remarks.}  The previously known bijection between $\T_n$ and $\RS_n$, given by Maria Monks \cite{Stan-2}, makes use of flip equivalence classes of increasing binary trees on vertex set $[n+1]$. By her method the permutation that corresponds to a tree $T\in\T_n$ is essentially determined by all the vertices of $T$ but the greatest vertex, while by our method the requested permutation $\phi(T)$ is determined by all of the non-root vertices of $T$.

\medskip
\section{Consequences of the bijection $\phi$}
In this section, with the benefits of the bijection $\phi$ we enumerate some families of pattern-avoiding simsun and double simsun permutations.

\medskip
\subsection{Restricted to $\RS_n(213)$}

Let $\X_n$ be the set of (unlabeled) ordered  1-2 trees with $n+1$ vertices. (The order of the subtrees of a vertex is significant.) It is known that $|\X_n|=M_n$ is the $n$th Motzkin number. For convenience, each tree $T\in\X_n$ is uniquely assigned a vertex-labeling by traversing $T$ in {\em right-to-left preorder} and labeling the vertices from $0$ to $n$. For example, $\X_3$ consists of the first four trees shown in Figure \ref{fig:Euler-4}, for which the vertices are increasing in right-to-left preorder, while for the fifth one they are not.
Clearly, the resulting trees are in the canonical form. The following result is an immediate consequence of the bijection $\phi$ when the map $\phi$ restricts to $\X_n$.

\begin{thm} \label{thm:right-to-left}
\[
|\RS_n(213)|=M_n.
\]
\end{thm}

\begin{proof} We shall prove that the map $\phi$ induces a bijection between $\X_n$ and $\RS_n(213)$.

Given a $T\in\X_n$, suppose the permutation $\phi(T)$ contains 213-patterns. Let $(x,y,z)$ be the 213-pattern with the least element $x$. If there is more than one choice, choose the first one (in lexicographic order). By the vertex-labeling of $T$, the vertex $z$ is a descendant of $y$. Moreover, $x$ is the sibling of $y$ if $y$ is in the rightmost path, and $x$ is the left child of $y$ otherwise.
In either case the right-to-left preorder $y,z,x$ of these vertices are not increasing, a contradiction. Hence $\phi(T)\in\RS_n(213)$.

On the other hand, given a $\sigma\in\RS_n(213)$, suppose the vertices of $\phi^{-1}(\sigma)$ are not increasing in right-to-left preorder. Let $(v,z)$ be the first pair of  consecutive vertices in this order such that $z\ge v+2$. Let $x=z-1$ and let $y$ be the parent of $z$. We observe that $x$ is not a descendant of $y$, and $(x,y,z)$ forms a 213-pattern in $\sigma$, a contradiction. Hence $\phi^{-1}(\sigma)\in\X_n$. The proof is completed.
\end{proof}

\smallskip
Let $\M_n$ denote the set of lattice paths, called {\em Motzkin paths} of length $n$, from the origin to the point $(n,0)$ using {\em up step} $\U=(1,1)$, {\em down step} $\D=(1,-1)$, and {\em level step} $\L=(1,0)$ that never pass below the $x$-axis. It is known that $|\M_n|=M_n$.
Now we establish a bijection $\chi:\X_n\rightarrow\M_n$. For an ordered 1-2 tree $T$, we associate $T$ with a word $\chi(T)$ of length $|T|-1$ with alphabet $\{\U,\L,\D\}$ by the following algorithm.

\smallskip
\noindent{\bf Algorithm C.}
\begin{enumerate}
\item[(C1)] If $T$ consists of the root vertex then $T$ is associated with an empty word.
\item[(C2)] Otherwise the word $\chi(T)$ is defined inductively by the following factorization.
\end{enumerate}
\begin{itemize}
\item If the root of $T$ has only one child $x$ then let
\[
\chi(T)=\L\cdot\chi(T'),
\]
where $T'=\tau(x)$ is the subtree of $T$ rooted at $x$.
\item If the root of $T$ has two children $u$ and $v$,  where $u$ (resp. $v$) is the left (resp. right) child, then let
\[
\chi(T)=\U\cdot\chi(T_1)\cdot\D\cdot\chi(T_2),
\]
where $T_1=\tau(v)$ and $T_2=\tau(u)$ are the right and left subtrees of the root of $T$, respectively.
\end{itemize}

\smallskip
\begin{exa} \label{exa:213-to-Motzkin-inversion} {\rm Take the tree
$T$ shown in Figure \ref{fig:231-Motzkin}(a).
Since the root has two children $(u,v)=(5,1)$, the path $\chi(T)$ can be factorized as $\chi(T)=\U\chi(T_1)\D\chi(T_2)$, where $T_1=\tau(1)$ and $T_2=\tau(5)$. Inductively, $\chi(T_1)=\U\D\L$ and $\chi(T_2)=\U\L\D$. The corresponding path $\chi(T)$ is shown in Figure \ref{fig:231-Motzkin}(b), where the labels of the steps indicate the corresponding vertices in $T$.
}
\end{exa}

\begin{figure}[ht]
\begin{center}
\includegraphics[width=3.5in]{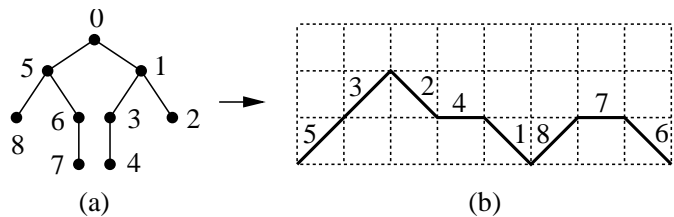}
\end{center}
\caption{\small An ordered 1-2 tree and the corresponding Motzkin path.} \label{fig:231-Motzkin}
\end{figure}

To find $\chi^{-1}$, given a Motzkin path $\pi$, we start with a root vertex $0$ and recover the tree $\chi^{-1}(\pi)$, rooted at $0$, inductively by a reverse procedure.

\smallskip
\noindent{\bf Algorithm D.}
\begin{enumerate}
\item[(D1)] If $\pi$ is empty then we associate $\pi$ with the root vertex.
\item[(D2)] Otherwise the tree $\chi^{-1}(\pi)$ is defined inductively, according to the following cases.
\end{enumerate}
\begin{itemize}
\item If $\pi$ starts with a level step, then we factorize $\pi$ as $\pi=\L\cdot\pi'$. Attach a vertex, say $x$, to the root and construct the subtree $\tau(x)=\chi^{-1}(\pi')$, root at $x$.
\item If $\pi$ starts with an up step, then we factorize $\pi$ as $\pi=U\pi_1D\pi_2$, where $U$ is the first step, $D$ is the first down step returning the $x$-axis, and $\pi_1$, $\pi_2$ are Motzkin paths of certain length (possibly empty). Attach a vertex $u$ (resp. $v$) as the left (resp. right) child of the root, and construct the subtrees $\tau(v)=\chi^{-1}(\pi_1)$ rooted at $v$ and $\tau(u)=\chi^{-1}(\pi_2)$ rooted at $v$.
\end{itemize}

\smallskip
The bijection $\chi:\X_n\rightarrow\M_n$ is established. Hence along with the map in the proof of Theorem \ref{thm:right-to-left}, we have the following result.

\begin{cor} \label{cor:213-Motzkin}
The map $\chi\circ\phi^{-1}$ is a bijection between the two sets $\RS_n(213)$ and $\M_n$.
\end{cor}

\medskip
\subsection{Restricted to $\DRS_n(213)$}
As mentioned by Callan in \cite[A004148]{Sloa}, the $n$th secondary structure number $S_n$ coincides with the number of Motzkin paths in $\M_n$ without consecutive up steps (or equivalently, without consecutive down steps). These paths are called {\em $\U\U$-free} (resp. {\em $\D\D$-free}). The map $\chi\circ\phi^{-1}$ established above has the following consequence when restricted to the set $\DRS_n(213)$.

\begin{thm} \label{thm:DD-free}
\[
|\DRS_n(213)|=S_n.
\]
\end{thm}

To prove this theorem, we characterize the permutations in $\DRS_n(213)$ among the permutations in $\RS_n(213)$, and show that the permutations corresponding to the $\D\D$-free paths in $\M_n$ have the same characterization.

The following proposition gives a sufficient condition for determining double simsun permutations among simsun permutations.

\smallskip
\begin{pro} \label{pro:inverse-simsun}
For a $\sigma\in\RS_n$, $\sigma^{-1}$ is simsun if either $\sigma$ has no 4132-patterns, or every 4132-pattern of $\sigma$ is contained in a 51342-pattern.
\end{pro}

\begin{proof}
Let $\sigma=\sigma_1\cdots\sigma_n\in\RS_n$ with $\sigma^{-1}\not\in\RS_n$. It suffices to prove that $\sigma$ has a 4132-pattern that is not contained in any 51342-pattern.

Let $t$ be the least integer such that the subword of $\sigma^{-1}$ restricted to $\{1,\dots,t\}$ contains a double descent $(\sigma^{-1}_i,\sigma^{-1}_{j},\sigma^{-1}_{k}$), where $i<j<k$ and $\sigma^{-1}_i=t$. Let $(\sigma^{-1}_{j},\sigma^{-1}_k)=(s,r)$. It follows that $(\sigma_r,\sigma_s,\sigma_t)=(i,j,k)$ is a decreasing triple of $\sigma$ restricted to $\{1,\dots,k\}$. The relative orders of these elements in $\sigma^{-1}$ and $\sigma$ are shown by the diagrams in Figure \ref{fig:sigma}. Since $\sigma\in\RS_n$, $(\sigma_r,\sigma_s,\sigma_t)$ is not a double descent. There are two cases.

(i) There exists an element $\sigma_p=q$ such that $q<j$ and $r<p<s$. Note that if $i<q<j$ then $\sigma^{-1}_q=p<s=\sigma^{-1}_j$, which is against the condition that $(\sigma^{-1}_i,\sigma^{-1}_{j},\sigma^{-1}_{k})$ is a double descent in the subword of $\sigma^{-1}$ restricted to $\{1,\dots,t\}$. Hence $q<i$, and
the quadruple $(\sigma_r,\sigma_p,\sigma_s,\sigma_t)$ is a 4132-pattern in $\sigma$ (see Figure \ref{fig:sigma}).

(ii) There exists an element $\sigma_p=q$ such that $j<q<k$ and $s<p<t$. Then $\sigma^{-1}_j<\sigma^{-1}_q<t$, which is also against the above condition. This eliminates case (ii).

Moreover, if the 4132-pattern $(\sigma_r,\sigma_p,\sigma_s,\sigma_t)$ is contained in a 51342-pattern then there exists an element $\sigma_g=h$ such that $j<h<k$ and $s<g<t$. This possibility is eliminated as in case (ii). The proof is completed.
\end{proof}

\begin{figure}[ht]
\begin{center}
\scalebox{1.25}{\input{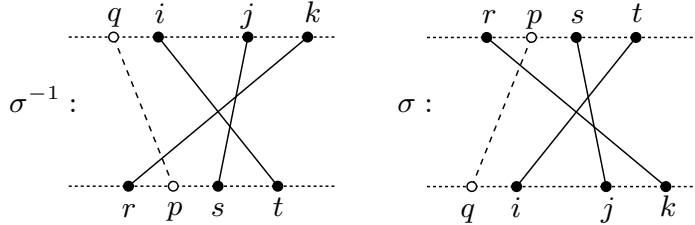}}
\end{center}
\caption{\small The relative order of $(\sigma_r,\sigma_p,\sigma_s,\sigma_t)$ for the permutations $\sigma^{-1}$ and $\sigma$.} \label{fig:sigma}
\end{figure}

Note that a double simsun permutation does not necessarily satisfy the condition in Proposition \ref{pro:inverse-simsun}. For example, the word 35142 is double simsun with a 4132-pattern 5142 not contained in any 51342-pattern.
However, when restricted to 213-avoiding simsun permutations the condition in Proposition \ref{pro:inverse-simsun} turns out to be the necessary condition.

\smallskip
\begin{thm} \label{thm:characterization}
For a $\sigma\in\RS_n(213)$, $\sigma^{-1}$ is simsun if and only if either $\sigma$ has no 4132-patterns, or every 4132-pattern of $\sigma$ is contained in a 51342-pattern.
\end{thm}

\begin{proof} The `if' part follows from Proposition \ref{pro:inverse-simsun}. For the `only if' part,
given a $\sigma\in\RS_n(213)$ with $\sigma^{-1}\in\RS_n$, it suffices to prove that if $\sigma$ has a 4132-pattern, say $(\sigma_m,\sigma_i,\sigma_j,\sigma_k)$, then the quadruple is contained in a 51342-pattern.

Let $(\sigma_m,\sigma_i,\sigma_j,\sigma_k)=(t,z,s,r)$. The relative orders of these elements are shown by the diagrams in Figure \ref{fig:51342-pattern}. Then ($\sigma^{-1}_r,\sigma^{-1}_s,\sigma^{-1}_t)$ forms a decreasing triple in the subword of $\sigma^{-1}$ restricted to $\{1,\dots,k\}$. Since $\sigma^{-1}$ is simsun, $(\sigma^{-1}_r,\sigma^{-1}_s,\sigma^{-1}_t)$ is not a double descent. There are two cases.

(i) There is an element $\sigma^{-1}_p=q$ such that $q<j$ and $r<p<s$. Note that if $q<i$ then $(\sigma_q,\sigma_i,\sigma_j)$ is a 213-pattern in $\sigma$, a contradiction. Hence $i<q<j$, and the quintuple  $(\sigma_m,\sigma_i,\sigma_q,\sigma_j,\sigma_k)$ forms a 51342-pattern (see Figure \ref{fig:51342-pattern}).

(ii) There is an element $\sigma^{-1}_p=q$ such that $j<q<k$ and $s<p<t$. Then the quintuple  $(\sigma_m,\sigma_i,\sigma_j,\sigma_q,\sigma_k)$ forms a 51342-pattern.

In either case the quadruple $(\sigma_m,\sigma_i,\sigma_j,\sigma_k)$ is contained in a 51342-pattern. The proof is completed.
\end{proof}

\begin{figure}[ht]
\begin{center}
\scalebox{1.2}{\input{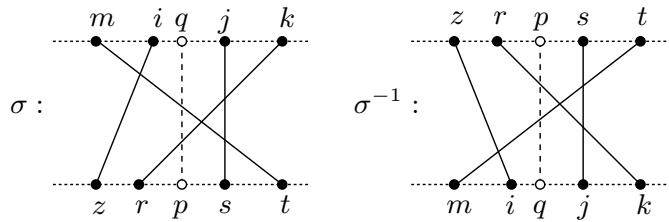}}
\end{center}
\caption{\small The diagram for a 51342-pattern in a permutation.} \label{fig:51342-pattern}
\end{figure}

In fact, the above characterization for $\DRS_n(213)$ is inspired by the following observation from the map $\chi\circ\phi^{-1}:\RS_n(213)\rightarrow\M_n$.

\begin{lem} \label{lem:51342} For a $\sigma\in\RS_n(213)$, the path $\chi\circ\phi^{-1}(\sigma)$ has consecutive down steps if and only if $\sigma$ has a 4132-pattern that is not contained in any 51342-pattern.
\end{lem}

\begin{proof} Given a $\sigma\in\RS_n(213)$, find the corresponding tree $T=\phi^{-1}(\sigma)\in\X_n$ and path $\pi=\chi(T)\in\M_n$.

Suppose $\pi$ contains consecutive down steps, say $D_2D_1$, let
$U_2$ and $U_1$ be their matching up steps, and let $u_1,v_1,u_2,v_2\in [n]$ be the vertices in $T$ associated with the steps $U_1,D_1,U_2,D_2$, accordingly. Then $u_1>v_1$ are siblings in $T$, and $u_2>v_2$ are siblings in the subtree $\tau(v_1)$, as shown in Figure \ref{fig:4132-pattern}. In particular, $u_2$ must be a leaf since $D_2$ and $D_1$ are consecutive in $\pi$. By the map $\phi$ we observe that $(u_1,v_1,u_2,v_2)$ forms a 4132-pattern in permutation $\sigma$. Moreover, the quadruple is not contained in any 51342-pattern since $u_2$ is a leaf in $T$.

On the other hand, suppose $\sigma$ has a 4132-pattern that is not contained in any 51342-pattern. Let $(w,x,y,z)\subseteq\sigma$ be the 4132-pattern with the least element $w$, which is not contained in any 51342-pattern. If there is more than one choice, choose the first one. Then the vertices $y,z$ are descendants of $x$ in $T$. Moreover, $w$ (resp. $y$) is the sibling of $x$ (resp. $z$) if $x$ (resp. $z$) is in the rightmost path, and  $w$ (resp. $y$) is the left child of $x$ (resp. $z$) otherwise. In either case $w$ and $y$ are left-child vertices. Let $D_1$ and $D_2$ be the down steps in $\pi$ associated with the siblings of $w$ and $y$, respectively.
Since $(w,x,y,z)$ is not contained in any 51342-pattern, $y$ is a leaf in $T$, and hence the two down steps $D_2$ and $D_1$ are consecutive in $\pi$.
\end{proof}

\begin{figure}[ht]
\begin{center}
\scalebox{1.2}{\input{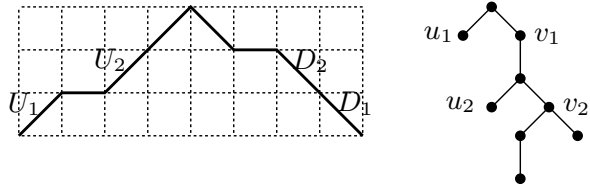}}
\end{center}
\caption{\small Two consecutive down steps and the corresponding vertices.} \label{fig:4132-pattern}
\end{figure}

Now we are able to prove Theorem \ref{thm:DD-free}.

\medskip
\noindent{\em Proof of Theorem \ref{thm:DD-free}.} We shall prove that the map $\chi\circ\phi^{-1}$ in Corollary \ref{cor:213-Motzkin} induces a one-to-one correspondence between the permutations in $\DRS_n(213)$ and the $\D\D$-free paths in $\M_n$.

Given a $\sigma\in\DRS_n(213)$, let $\pi=\chi\circ\phi^{-1}(\sigma)\in\M_n$. It follows from Theorem \ref{thm:characterization} and Lemma \ref{lem:51342} that $\pi$ is $\D\D$-free.

On the other hand, given a $\D\D$-free path $\pi\in\M_n$, find the corresponding tree $T=\chi^{-1}(\pi)\in\X_n$ and permutation $\sigma=\phi(T)\in\RS_n(213)$. There are two cases.

(i) The height of $\pi$ is at most 1. Then there are no quadruples $u_1,v_1,u_2,v_2\in [n]$ such that $u_1,v_1$ are siblings in $T$ and $u_2,v_2$ are siblings in the subtree $\tau(v_1)$. Then $\sigma$ has no 4132-patterns.

(ii) The height of $\pi$ of is at least 2. By Lemma \ref{lem:51342}, every 4132-pattern in $\sigma$ is contained in a 51342-pattern.

By Theorem \ref{thm:characterization}, $\sigma^{-1}$ is simsun. Hence $\sigma\in\DRS_n(213)$. The proof is completed. \qed

\smallskip
Next we study the $\D\D$-free Motzkin paths that correspond to the permutations in $\DRS_n(132,213)$.

\medskip
\subsection{Restricted to $\DRS_n(132,213)$} Consider the subset $\R_{n}\subseteq\M_{n}$ of Motzkin paths of height 1 with no level steps on the $x$-axis. For example, $\R_5=\{\U\L\L\L\D,\U\L\D\U\D,\U\D\U\L\D\}$. The $n$th Fibonacci number $F_{n}$ also counts the number of {\em compositions} of $n+1$ with no part equal to 1 (see \cite[A000045]{Sloa}). For example, $\{5,3+2,2+3\}$ are the requested compositions of 5. It is clear that $|\R_{n+2}|=F_{n+1}$ since the paths in $\R_{n+2}$ are uniquely determined by their block-sizes, which are identical to the compositions of $n+2$.

For a path $\pi\in\R_{n+2}$, we factorize $\pi$ into blocks as $\pi=\mu_1\cdots\mu_j$, and then form a new path
\begin{equation}
\pi'=\mu_1\cdots\mu_{j-1}\mu'_j
\end{equation}
from $\pi$ by replacing the last block $\mu_j$ by $\mu'_j$, where $\mu'_j$ is the remaining part of $\mu_j$ when the first step and the last step are removed. Note that $\mu'_j$ is a segment of level steps (possibly empty) on the $x$-axis, and $\pi'\in\M_n$. Define $\Q_n=\{\pi':\pi\in\R_{n+2}\}$.  For example, $\Q_3=\{\L\L\L,\U\L\D,\U\D\L\}$. Clearly, $|\Q_n|=|\R_{n+2}|$.

By the {\em leftmost path} of an ordered 1-2 tree we mean the maximal path $u_0,u_1,\dots,u_k$ such that $u_0=0$ is the root and $u_i$ is the left child of $u_{i-1}$, for $1\le i\le k$. For the paths $\pi'$ in $\Q_n$, a characterization of the corresponding tree $\chi^{-1}(\pi')\in\X_n$ is that if a vertex $x$ has two children then the vertex $x$ is in the leftmost path.

\begin{thm} \label{thm:Qn}
\[
|\DRS_n(132,213)|=F_{n+1}.
\]
\end{thm}

\begin{proof} We shall prove that
the map $\chi\circ\phi^{-1}$ induces a bijection between $\DRS_n(132,213)$ and $\Q_n$.

Given a $\sigma\in\DRS_n(132,213)$, find the corresponding tree $T=\phi^{-1}(\sigma)\in\X_n$. Suppose there is a vertex $x$ in $T$ with two children $y>z$ such that $x$ is not in the leftmost path. We observe that if
$x$ is in the rightmost path then the triple $(x,y,z)$ is a 132-pattern in $\sigma$, otherwise the parent of $x$, say $w$, together with $y$ and $z$ form a 132-pattern $(w,y,z)$ in $\sigma$, which is against the 132-avoiding property of $\sigma$. Hence $\chi\circ\phi^{-1}(\sigma)\in\Q_n$.

On the other hand, given a $\pi'\in\Q_n$, find the corresponding tree $T=\chi^{-1}(\pi')\in\X_n$ and permutation $\sigma=\phi(T)\in\RS_n(213)$. Since $\pi'$ is $\D\D$-free, $\sigma\in\DRS_n(213)$.
Suppose $\sigma$ has 132-patterns, let $(x,y,z)$ be the first 132-pattern in $\sigma$. Then the vertices $y,z$ are descendants of $x$ in $T$, and $y$ is either the left child or the sibling of $z$. In either case there is a vertex with two children in the subtree $\tau(x)$, which is against the characterization of $T$. Hence $\sigma\in\DRS_n(132,213)$.
\end{proof}

\smallskip
For the permutations $\sigma\in\DRS_n(132,213)$, by the above argument, the corresponding trees can be partitioned into paths whose starting points are the ones in the leftmost path. Hence $\sigma$ can be factorized into maximal increasing subwords (by putting a dot between every descent pair), which can be constructed as follows.

Let $\beta=n(n-1)\cdots 1$ be the word with $\beta_i=n+1-i$, for $1\le i\le n$. For a composition $C=\{t_1,\dots,t_j\}$ of $n$ with $t_i\ge 2$ ($2\le i\le j-1$) and $t_1,t_j\ge 1$, we factorize $\beta$ with respect to $C$ as $\beta=\mu_1\mu_2\cdots\mu_j$ so that the $i$th subword $\mu_i$ is of length $t_i$. We associate $C$ with a permutation $\zeta(C)\in\DRS_n(132,213)$ defined by
\begin{equation}
\zeta(C)=\nu_1\nu_2\cdots\nu_j,
\end{equation}
where $\nu_i$ is the word in reverse order of $\mu_i$.

For example, take a composition $C=(1,3,2,3)$ of 9. The factorization of $\beta$ and the associated permutation $\zeta(C)$ are shown below.
\[
C=(1,3,2,3)\quad\longleftrightarrow\quad\beta=9.876.54.321\quad\longleftrightarrow\quad\zeta(C)=967845123.
\]

\medskip
\subsection{Restricted to $\RS_n(231)$} Based on the bijection $\phi:\T_n\rightarrow\RS_n$, we shall establish a connection between $\RS_n(231)$ and $\RS_n(213)$.

\begin{thm} \label{thm:bijection-231-213}
There is a bijection between the two sets $\RS_n(231)$ and $\RS_n(213)$.
\end{thm}

\smallskip
With the map $\phi$ in Theorem \ref{thm:phi} we consider the set $\F_n=\{\phi^{-1}(\sigma)|\sigma\in\RS_n(231)\}$ of trees that correspond to 231-avoiding simsun permutations. Note that the vertices of these trees are not necessarily increasing in right-to-left preorder. Moreover, these trees satisfy the condition that if a vertex has two children then its left child is a leaf. With this condition and 231-avoiding property,
we establish a bijection $\psi:\F_n\rightarrow\X_n$, for which each tree $T\in\F_n$ is uniquely transformed into a tree $\psi(T)\in\X_n$ by the following switching process.

\smallskip
\noindent{\bf Algorithm E.}
\begin{enumerate}
\item[(E1)] Traverse $T$ in right-to-left preorder, and find the first pair $(v,z)$ of consecutive vertices such that $z\ge v+2$. Then $z$ is the right child of $v$  (due to 231-avoiding). Moreover, if $v$ has more than one child, let $y$ be the left child of $v$. One can always find the leaf $x$ with $x=z-1$ to form a new tree $T'$ by switching the subtrees $\tau(y)$ and $\tau(z)$ from $v$ to $x$.
\item[(E2)] If the vertices of $T'$ are increasing in right-to-left preorder then we are done. Otherwise go to (E1) and proceed to process $T'$.
\end{enumerate}

\smallskip
\begin{exa} \label{exa:231-to-213}  {\rm
Take a permutation $\sigma=51324867\in\RS_8(231)$. The tree $T=\phi^{-1}(\sigma)$ is shown as Figure \ref{fig:231-213-trees}(a). The first consecutive pair $(x,v)$ in  right-to-left preorder such that $z\ge v+2$ is $(v,z)=(2,4)$. Then $T'$ is obtained by switching $\tau(4)$ to the leaf $x=3$, shown as Figure \ref{fig:231-213-trees}(b). Repeating the procedure, one obtains the requested tree $\psi(T)\in\X_n$ from $T'$ by switching the subtrees $\tau(8)$ and $\tau(6)$ to the leaf $5$, shown as Figure \ref{fig:231-213-trees}(c).
}
\end{exa}

\begin{figure}[ht]
\begin{center}
\includegraphics[width=3.5in]{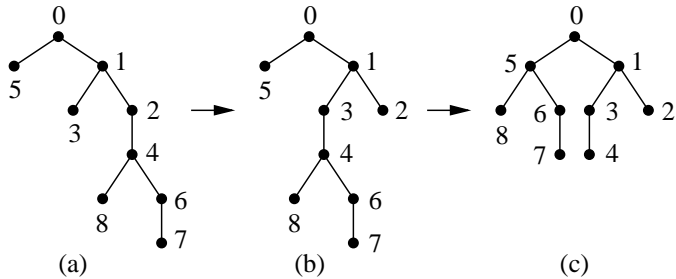}
\end{center}
\caption{\small Transforming an increasing tree into right-to-left preorder.} \label{fig:231-213-trees}
\end{figure}

To find $\psi^{-1}$, given a $T\in\X_n$, we shall recover the tree $\psi^{-1}(T)\in\F_n$ by a reverse process.

\smallskip
\noindent{\bf Algorithm F.}

\begin{enumerate}
\item[(F1)] Traverse $T$ (from left to right) by a depth-first-search, and find the first left-child vertex $x$ such that $x$ is not a leaf. Let $w$ be the sibling of $x$, and let $y,z$ be the children of $x$, where $y$ is empty if $x$ has only one child. Find the greatest leaf $v$ in the subtree $\tau(w)$, and form a new tree $T'$ by switching $\tau(y)$ and $\tau(z)$ from $x$ to $v$.
\item[(F2)] If all of the left-child vertices of $T'$ are leaves then we are done. Otherwise go to (F1) and proceed to process $T'$.
\end{enumerate}

This establishes the bijection $\psi:\F_n\rightarrow\X_n$. By Theorem \ref{thm:phi}, we prove that
the map $\phi\circ\psi\circ\phi^{-1}$ induces a bijection between the two sets $\RS_n(231)$ and $\RS_n(213)$.
This completes the proof of Theorem \ref{thm:bijection-231-213}.

\smallskip
Moreover, with the map $\chi:\X_n\rightarrow\M_n$ the composite $\chi\circ\psi\circ\phi^{-1}$ establishes a connection between the permutations in $\RS_n(231)$ and the paths in $\M_n$.

\smallskip
\begin{cor} \label{cor:231-Motzkin}
The map $\chi\circ\psi\circ\phi^{-1}$ induces a bijection between the two sets $\RS_n(231)$ and $\M_n$.
\end{cor}

\begin{exa} {\rm
Take again $\sigma=51324867\in\RS_8(231)$. As shown in Example \ref{exa:231-to-213}, the tree $\phi^{-1}(\sigma)\in\F_n$ is transformed into a tree $\psi\circ\phi^{-1}(\sigma)\in\X_n$, see Figure \ref{fig:231-213-trees}. Following Example \ref{exa:213-to-Motzkin-inversion}, we obtain the corresponding Motzkin path $\chi\circ\psi\circ\phi^{-1}(\sigma)$, see Figure \ref{fig:231-Motzkin}(b). Note that the number of inversions of $\sigma$ equals the area under the path $\chi\circ\psi\circ\phi^{-1}(\sigma)$.
}
\end{exa}

\smallskip
\noindent{\bf Remarks.} 1. The bijection $\chi\circ\phi^{-1}:\RS_n(213)\rightarrow\M_n$ (in Corollary \ref{cor:213-Motzkin}) is equivalent to the one given by Deutsch and Elizalde in the third proof of \cite[Proposition 4.1]{DeutEliz}, for which the Motzkin paths obtained by their method are in reverse order.

2. The bijection $\chi\circ\psi\circ\phi^{-1}:\RS_n(231)\rightarrow\M_n$ (in Corollary \ref{cor:231-Motzkin}) is equivalent to the one given in the third proof of \cite[Proposition 5.1]{DeutEliz}.

\medskip
\subsection{Restricted to $\RS_n(231,213)$} We consider the paths that correspond to the fixed points of the map $\psi:\F_n\rightarrow\X_n$, i.e., $T=\psi(T)$.

A {\em weak ascent} in a Motzkin path is a maximal sequence of consecutive up steps and level steps. Let $\W_n\subseteq\M_n$ be the set of Motzkin paths with exactly one weak ascent. For example, $\W_4=\{\L\L\L\L,\L\L\U\D,\L\U\L\D,\U\L\L\D,\U\U\D\D\}$. Note that $|\W_n|=F_{n+1}$, as mentioned by Deutsch in \cite[A000045]{Sloa}.
For the paths $\pi$ in $\W_n$, the corresponding trees $\chi^{-1}(\pi)\in\X_n$ can be characterized as ordered 1-2 trees whose left-child vertices are all leaves.

\smallskip
\begin{thm}
\[
|\RS_n(231,213)|=F_{n+1}.
\]
\end{thm}

\begin{proof} We shall prove that map $\chi\circ\phi^{-1}$ induces a bijection between  $\RS_n(231,213)$ and $\W_n$.

Given a $\sigma\in\RS_n(231,213)$, find the corresponding tree $T=\phi^{-1}(\sigma)\in\X_n$. Suppose $T$ has is a left-child vertex $x$ that is not a leaf, let $z$ be the sibling of $x$ and let $y$ be the right child of $x$. Then the triple $(x,y,z)$ is a 231-pattern in $\sigma$, a contradiction. Hence $\chi\circ\phi^{-1}(\sigma)\in\W_n$.

On the other hand, given a $\pi\in\W_n$, find the corresponding tree $T=\chi^{-1}(\pi)\in\X_n$ and permutation $\sigma=\phi(T)\in\RS_n(213)$.
Suppose $\sigma$ has 231-patterns, let $(x,y,z)$ be the 231-pattern with the least element $x$. If there is more than one choice, choose the first one. Then the vertex $x$ is either the left child or the sibling of $z$ in $T$. Moreover, $y$ is the right child of $x$. Thus $x$ is a left-child vertex with children, which is against the characterization of $T$. Hence $\sigma\in\RS_n(231,213)$.
\end{proof}


\section{Enumeration of $\DRS_n(\omega)$ for other $\omega$-patterns in $\mathfrak{S}_3$}
In this section we enumerate pattern-avoiding double simsun permutations for the other patterns of length 3.

\smallskip
\subsection{Avoiding 132}
To enumerate $\DRS_n(132)$, we establish a connection to $\DRS_n(213)$. For this purpose we consider an involution $\Gamma:\mathfrak{S}_n\rightarrow\mathfrak{S}_n$ as follows. Let $\sigma=\sigma_1\cdots\sigma_n\in\mathfrak{S}_n$, and let $\Gamma(\sigma)=\omega_1\cdots\omega_n$, where $\omega_i$ is defined by
\begin{equation}
\omega_i=n+1-\sigma^{-1}_{n+1-i},
\end{equation}
for $1\le i\le n$. Note that $\Gamma(\sigma^{-1})=\Gamma(\sigma)^{-1}$. Let $\DRS_n$ be the set of double simsun permutations in $\mathfrak{S}_n$.
For a $\sigma\in\DRS_n$, the following proposition gives a sufficient condition for determining $\Gamma(\sigma)$ being double simsun.

\medskip
\begin{pro} \label{pro:35142-pattern}
Given a $\sigma\in\DRS_n$, if $\Gamma(\sigma)$ is not double simsun then $\Gamma(\sigma)$ contains a 42513-pattern, or equivalently, $\sigma$ contains a 35142-pattern.
\end{pro}

\begin{proof} Let $\sigma=\sigma_1\cdots\sigma_n\in\DRS_n$,  and let $\Gamma(\sigma)=\omega=\omega_1\cdots\omega_n\not\in\DRS_n$. Then either $\omega\not\in\RS_n$ or $\omega^{-1}\not\in\RS_n$. Suppose $\omega\not\in\RS_n$, we claim that $\omega$ contains a $42513$-pattern.

Let $t$ be the least integer such that the subword of $\omega$ restricted to $\{1,\dots,t\}$ contains a double descent $(\omega_i,\omega_j,\omega_k)$, where $i<j<k$ and $\omega_i=t$. Since $\sigma\in\DRS_n$, the following observations hold.

(i) $(j,k)\neq(i+1,i+2)$. Otherwise, $\sigma^{-1}$ contains a double descent \[\sigma^{-1}_{n+1-k}>\sigma^{-1}_{n+1-j}>\sigma^{-1}_{n+1-i}.\]

(ii) $(\omega_j,\omega_k)\neq(t-1,t-2)$. Otherwise, $\sigma$ contains a double descent \[\sigma_{n-t+1}>\sigma_{n-t+2}>\sigma_{n-t+3}.\]
It follows from (i) and (ii) that
there exists an element $\omega_m$ such that $\omega_m>\omega_i$ and $j<m<k$.
Now, we consider the following two cases.

Case I. $\omega_j\neq t-1$. Let $\omega_g=t-1$. Then $\omega_g$ must appear after $\omega_k$, i.e., $g>k$. The reason is that if $\omega_g$ appears before $\omega_j$ (i.e., $g<j$) then the triple $(\omega_g,\omega_j,\omega_k)$ forms a double descent in the subword of $\omega$ restricted to $\{1,\dots,t-1\}$, which is against the choice of $t$. Moreover, if $\omega_g$ appears between $\omega_j$ and $\omega_k$ (i.e., $j<g<k$) then this contradicts that $(\omega_i,\omega_j,\omega_k)$ is a double descent in the subword restricted to $\{1,\dots,t\}$. Hence the quintuple $(\omega_i,\omega_j,\omega_m,\omega_k,\omega_g)$ forms a 42513-pattern in $\omega$.

Case II. $\omega_j=t-1$. Find the greatest element $r$ after $\omega_j$ such that $r<\omega_j$, say $r=\omega_g$. Note that $\omega_g\ge\omega_k$. The relative order of $\omega_i,\omega_j,\omega_k$, and $\omega_g$ in $\omega$ is shown by the diagram on the left of Figure \ref{fig:Gamma}. For convenience, let $z'=n+1-z$ for every $z\in [n]$. Note that the relative order of the corresponding elements in $\sigma$ is shown by the diagram on the right of Figure \ref{fig:Gamma}, which is obtained from the left one by flipping vertically and then horizontally.
We observe that the triple $(\sigma_{n-t+1},\sigma_{n-t+2},\sigma_{r'})$ always forms a double descent in the subword of $\sigma$ restricted to $\{1,\dots,i'\}$. The only concern is that if there exists an element $\sigma_p=q$ such that $j'<q<i'$ and $n-t+2<p<r'$ then $(\sigma_{n-t+1},\sigma_{n-t+2},\sigma_{r'})$ will no longer be a double descent. However, this will not happen since $(\omega_i,\omega_j)$ is a descent in the subword of $\omega$ restricted to $\{1,\dots,t\}$ and thus there are no such elements $\omega_{q'}=p'$ in $\omega$ with $r<p'<t-1$ and $i<q'<j$.
That $\sigma$ contains a double descent is against the condition that $\sigma$ is simsun. This eliminates case II.

This proves that if $\omega\not\in\RS_n$ then $\omega$ contains a 42513-pattern.
On the other hand, if $\omega\in\RS_n$ then $\omega^{-1}\not\in\RS_n$. By the above argument, $\omega^{-1}$ contains a 42513-pattern and so does $\omega$.
By the definition of $\Gamma$, $\sigma$ contains a 35142-pattern if and only if $\omega$ contains a 42513-pattern. The proof is completed.
\end{proof}

\begin{figure}[ht]
\begin{center}
\scalebox{1.2}{\input{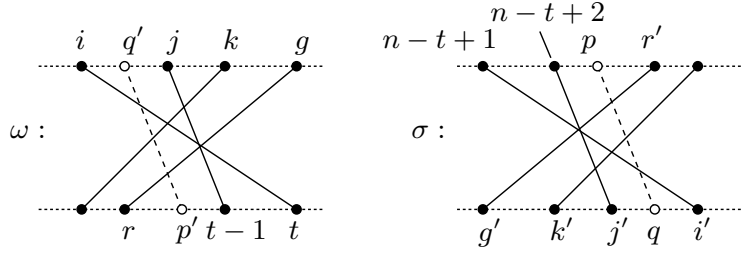}}
\end{center}
\caption{\small The relative order of $(\omega_i,\omega_j,\omega_k,\omega_g)$ for the permutations $\omega$ and $\sigma$.} \label{fig:Gamma}
\end{figure}

\begin{thm}
\[
|\DRS_n(132)|=S_n.
\]
\end{thm}

\begin{proof} We shall prove that $\Gamma$ induces a bijection between $\DRS_n(132)$ and $\DRS_n(213)$.

Given a $\sigma\in\DRS_n(132)$, let $\omega=\Gamma(\sigma)$. Since $\sigma$ is 132-avoiding, $\sigma$ contains no 35142-patterns and, by Proposition \ref{pro:35142-pattern}, $\omega$ is double simsun. Moreover, by the definition of $\Gamma$, $\omega$ is 213-avoiding.  Hence $\omega\in\DRS_n(213)$.

On the other hand, given an $\omega'\in\DRS_n(213)$, let $\sigma'=\Gamma^{-1}(\omega')$. Then $\sigma'=\Gamma(\omega')$. Since $\omega'$ is 213-avoiding, $\omega'$ contains no 35142-patterns and, by Proposition \ref{pro:35142-pattern}, $\sigma'$ is double simsun. Moreover, by the definition of $\Gamma$, $\sigma'$ is 132-avoiding. Hence $\sigma'\in\DRS_n(132)$.

This proves $|\DRS_n(132)|=|\DRS_n(213)|$. The assertion follows from Theorem \ref{thm:DD-free}.
\end{proof}

\smallskip
Along with the fact $|\RS_n(132)|=S_n$ \cite[Theorem 3.1]{DeutEliz}, we prove the relation
\begin{equation}
\DRS_n(132)=\RS_n(132).
\end{equation}

By the bijection $\Gamma$, it is clear that the permutations in $\DRS_n(132)$ and $\DRS_n(213)$ are equidistributed with respect to excedances (i.e., $\sigma_i>i$) and fixed points.

\smallskip
\begin{cor} \label{pro:312-231} The number of permutations in $\DRS_n(132)$ with $i$ excedances and $j$ fixed points  equals the number of permutations in $\DRS_n(213)$ with $i$ excedances and $j$ fixed points.
\end{cor}

\medskip
\subsection{Avoiding 312/231} The following relations show that the sets $\RS_n(312)$, $\DRS_n(312)$, and $\DRS_n(231)$ have the same cardinality.

\begin{lem} \label{lem:DRS321=RS312} The following facts hold.
\begin{enumerate}
\item $\DRS_n(312)=\RS_n(312)$.
\item The map $\sigma\rightarrow\sigma^{-1}$ is a bijection between $\DRS_n(312)$ and $\DRS_n(231)$.
\end{enumerate}
\end{lem}

\begin{proof} (i) It is clear that $\DRS_n(312)\subseteq\RS_n(312)$. On the other hand, we observe that
if $\sigma\in\RS_n(312)$ then $\sigma$ contains no 4132-patterns. It follows from Proposition \ref{pro:inverse-simsun} that $\sigma^{-1}$ is simsun. Hence $\sigma\in\DRS_n(312)$.

(ii) Note that for any $\sigma\in\mathfrak{S}_n$, $\sigma^{-1}$ is 231-avoiding if and only if $\sigma$ is 312-avoiding. By definition, it is clear that the map $\sigma\rightarrow\sigma^{-1}$ is a bijection between $\DRS_n(231)$ and $\DRS_n(312)$.
\end{proof}

\smallskip
Recall that the number $2^{n-1}$ counts the number of compositions
of $n$. For example, the compositions of 4 consist of
$\{4,3+1,2+2,1+3,2+1+1,1+2+1,1+1+2,1+1+1+1\}$. In the following we
present simple constructions for the permutations in $\DRS_n(312)$
and $\DRS_n(231)$, respectively, using compositions of $n$.

\begin{thm} \label{thm:DRS231=DRS312}
\[
|\DRS_n(231)|=|\DRS_n(312)|=|\RS_n(312)|=2^{n-1}.
\]
\end{thm}

\smallskip
Let $\alpha=1\cdots n$ be the word with $\alpha_i=i$, for $1\le i\le n$. For a composition $C=(t_1,t_2,\dots,t_j)$ of $n$, let $s_i=t_1+\cdots+t_i$, for $1\le i\le j-1$. For convenience, let $s_0=0$ and $s_j=n$.  We factorize $\alpha$ with respect to $C$ as $\alpha=\mu_1\mu_2\cdots\mu_j$ by putting a dot between $s_i$ and $s_i+1$, for $1\le i\le j-1$. Namely, the $i$th subword is $\mu_i=(s_{i-1}+1)\cdots s_i$. Then we associate $C$ with the permutation $\rho(C)\in\mathfrak{S}_n$ defined by
\begin{equation}
\rho(C)=\widehat{\mu}_1\widehat{\mu}_2\cdots\widehat{\mu}_j,
\end{equation}
where $\widehat{\mu}_i=(s_{i-1}+2)\cdots s_i(s_{i-1}+1)$ is the word obtained from $\mu_i$ by moving the first element $(s_{i-1}+1)$ to the end of $\mu_i$.

For example, take $C=(3,2,1,3)$, a composition of 9. The factorization of $\alpha$ and the associated permutation $\rho(C)$ are shown below.
\[
C=(3,2,1,3)\quad\longleftrightarrow\quad\alpha=123.45.6.789\quad\longleftrightarrow\quad\rho(C)=231546897.
\]

\smallskip
\begin{pro} \label{pro:rho} The map $\rho$ is a bijection between the set of compositions of $n$ and the set $\RS_n(312)$.
\end{pro}

\begin{proof} For a composition $C$ of $n$, it is straightforward to verify that the associated permutation $\rho(C)$ is 312-avoiding simsun.

On the other hand, given a $\sigma\in\RS_n(312)$, we observe that for all $k$, the subword $\omega_1\cdots\omega_k$ of $\sigma$ restricted to $\{1,\dots,k\}$ satisfies the condition that either $\omega_{k-1}=k$ or $\omega_k=k$. One can factorize $\sigma$ into subwords as follows. For $k=1,\dots,n$, check the subword $\omega_1\cdots\omega_k$ of $\sigma$ restricted to $\{1,\dots,k\}$, and put a dot in $\sigma$ at the end of $k$ if $\omega_k=k$. Then the lengths of these subwords form the requested composition $\rho^{-1}(\sigma)$ of $n$. The assertion follows.
\end{proof}

\smallskip
By Lemma \ref{lem:DRS321=RS312} and Proposition \ref{pro:rho}, the proof of Theorem \ref{thm:DRS231=DRS312} is completed.

\medskip
Since $\sigma\in\DRS_n(312)=\RS_n(312)$ if and only if $\sigma^{-1}\in\DRS_n(231)$, we have a similar construction for the permutations in $\DRS_n(231)$.

For a composition $C=(t_1,t_2,\dots,t_j)$ of $n$ and the factorization $\alpha=\mu_1\mu_2\cdots\mu_j$ of $\alpha$ with respect to $C$, the corresponding permutation $\varrho(C)\in\DRS_n(231)$ is defined by
\begin{equation}
\varrho(C)=\overline{\mu}_1\overline{\mu}_2\cdots\overline{\mu}_j,
\end{equation}
where $\overline{\mu}_i=s_i(s_{i-1}+1)\cdots(s_i-1)$ is obtained from $\mu_i$ by moving the lase element $s_i$ to the beginning of $\mu_i$.
It is clear that $\varrho(C)$ is the inverse of $\rho(C)$. We have the following result.

\begin{pro} \label{pro:varrho} The map $\varrho$ is a bijection between the set of compositions of $n$ and the set $\DRS_n(231)$.
\end{pro}

\smallskip
Counting the compositions of $n$ by the number of parts greater than 1 yields the following result.

\smallskip
\begin{cor}
The number of permutations in $\DRS_n(312)$, and respectively in $\DRS_n(231)$, with $k$ descents is $\binom{n}{2k}$.
\end{cor}

\begin{proof}
Note that the number of compositions $C$ of $n$ with $k$ parts greater than 1 is $\binom{n}{2k}$. By the construction of $\rho(C)$ and $\varrho(C)$, each of these parts contributes exactly one descent to the permutation $\rho(C)$ as well as $\varrho(C)$, and all the other parts contribute fixed points.
\end{proof}

\smallskip

The permutations in $\DRS_n(312)$ and $\DRS_n(231)$ are essentially equidistributed with respect to excedances and fixed points.

\smallskip
\begin{cor} \label{pro:312-231} The number of permutations in $\DRS_n(312)$ with $i$ excedances and $j$ fixed points  equals the number of permutations in $\DRS_n(231)$ with $n-i-j$ excedances and $j$ fixed points.
\end{cor}

\begin{proof} Given a $\sigma\in\DRS_n(312)$ with $i$ excedances, $j$ fixed points, and $k$ descents,  by the map $\rho^{-1}$, the permutation $\sigma$ can be factorized as $\sigma=\widehat{\mu}_1\widehat{\mu}_2\dots\widehat{\mu}_{j+k}$. Let $C=\rho^{-1}(\sigma)=(t_1,\dots,t_{j+k})$, where $t_d$ is the length of the $d$th subword $\widehat{\mu}_d$, for $1\le d\le j+k$.  Note that among these subwords, there are $j$ subwords of length 1 and $k$ subwords of length at least 2. Moreover, each $\widehat{\mu}_d$ contributes exactly $t_d-1$ excedances to $\sigma$. Hence $i+j+k=n$. By the map $\varrho$, the permutation $\varrho(C)=\overline{\mu}_1\overline{\mu}_2\dots\overline{\mu}_{j+k}\in\DRS_n(231)$ contains exactly $k$ excedances, $j$ fixed points, and $k$ descents since each subword $\overline{\mu}_d$ of length $t_d\ge 2$ contributes one excedance and one descent to $\varrho(C)$.
\end{proof}

\smallskip
By the bijections $\rho$ and $\varrho$, it is clear that the permutations in $\DRS_n(312,231)$ are in one-to-one correspondence with the words obtained by factorizing $\alpha$ into subwords of length 1 or 2, and then interchanging the letters in each subword of length 2.
This proves the following result.

\begin{cor}
\[
|\DRS_n(312,231)|=F_{n+1}.
\]
\end{cor}

\medskip
\subsection{Avoiding 321/123}
We make use of the results in \cite{DeutEliz} to enumerate the sets $\DRS_n(321)$ and $\DRS_n(123)$.

\begin{thm} We have
\begin{enumerate}
\item $|\DRS_n(321)|=C_n$,
\item $|\DRS_4(123)|=5$, $|\DRS_5(123)|=3$, and $|\DRS_n(123)|=2$ for $n\ge 6$.
\end{enumerate}
\end{thm}

\begin{proof} (i) Note that $|\RS_n(321)|=C_n$. By the same argument as in the proof of Lemma \ref{lem:DRS321=RS312}, we have $\DRS_n(321)=\RS_n(321)$.

(ii) Note that $\RS_4(123)=\{3412,4231,4132,3142,2413,2143\}$. By the proof of \cite[Proposition 2.1]{DeutEliz}, for $n\ge 4$, each permutation $\sigma=\sigma_1\cdots\sigma_n\in\RS_n(123)$ can produce a unique permutation in $\RS_{n+1}(123)$ by inserting $n+1$ between $\sigma_1$ and $\sigma_2$ if $\sigma_1>\sigma_2$ or to the left of $\sigma_1$ if $\sigma_1<\sigma_2$. To determine $\DRS_n(123)$, we shall eliminate those possibilities $\sigma\in\RS_n(123)$ such that $\sigma^{-1}$ is not simsun.

For $n\ge 4$, if the subword of $\sigma$ restricted to $\{1,2,3,4\}$ is 4132 then $\sigma^{-1}$ is not simsun since $\sigma^{-1}$ contains a double descent $(\sigma^{-1}_2,\sigma^{-1}_3,\sigma^{-1}_4)$. In particular, this eliminates 4132 from $\RS_4(123)$. One can check that the others are double simsun. Then $|\DRS_4(123)|=5$.

Since $\sigma\in\DRS_n(123)$ if and only if $\sigma^{-1}\in\DRS_n(123)$, the last element $\sigma_n$ of $\sigma$ is either 1 or 2. In particular, this eliminates 52413 and 25143 from $\RS_5(123)$, and $\DRS_5(123)=\{53412,45231,35142\}$. Moreover, for $n\ge 6$ if the subword of $\sigma$ restricted to $\{1,\dots,6\}$ is 635142 then $\sigma^{-1}$ is not simsun since $\sigma^{-1}$ contains a double descent $(\sigma^{-1}_4,\sigma^{-1}_5,\sigma^{-1}_6)$. In particular, this eliminates 635142 from $\RS_6(123)$, and $\DRS_6=\{563412,654231\}$. Note that each member of $\DRS_n(123)$ produces a unique permutation in $\DRS_{n+1}(123)$, for $n\ge 6$. The assertion follows.
\end{proof}


\section{Remarks}
In this paper we enumerate double simsun permutations that avoid a pattern of length 3. However, the total number of double simsun permutations in $\mathfrak{S}_n$ is still unknown, and the initial values 1, 2, 5, 15, 52, 204, 892, 4297 do not match any known integer sequence in Sloane's Encyclopedia \cite{Sloa}. In addition to this enumerative problem, we are also interested in an analogous characterization result for double simsun permutations as Theorem \ref{thm:characterization}, i.e., characterize double simsun permutations among simsun permutations by pattern-conditions.


\end{document}